\documentclass[12pt]{amsart}

\usepackage{a4wide}
\usepackage{amssymb}
\usepackage{latexsym}
\usepackage{amsmath}
\usepackage{amsthm}
\usepackage{amsfonts}
\usepackage{color}
\usepackage{epsfig}

\newcommand{\I}{\mathcal{I}}

\newcommand{\R}{\mathbb{R}}

\theoremstyle{plain}
\newtheorem{defi}{Definition}[section]

\newtheorem{teo}[defi]{Theorem}

\newtheorem{remark}[defi]{Remark}

\theoremstyle{definition}
\newtheorem{example}[defi]{Example}

\theoremstyle{remark}

\numberwithin{equation}{section}

\begin{document}

\title[]{Propagation of minima for nonlocal operators}

\author[]{Isabeau Birindelli}

\author[]{Giulio Galise}
\address{
Isabeau Birindelli and Giulio Galise:
Dipartimento di Matematica Guido Castelnuovo, Sapienza Universit\`a di Roma,
Piazzale Aldo Moro 5, Roma, ITALIA.
\newline \textit{Email address:} {\tt galise@mat.uniroma1.it, isabeau@mat.uniroma1.it}
}
\author[]{Hitoshi Ishii}
\address{
Hitoshi Ishii:
Institute for Mathematics and Computer Science, Tsuda University, Kodaira, Tokyo, Japan
\newline \textit{Email address:} {\tt hitoshi.ishii@waseda.jp}
}

\keywords{Nonlinear and degenerate integral operators, strong maximum principle, propagation of minima, positivity sets, viscosity solutions.
\indent I.B. and G.G were partially supported by INdAM-GNAMPA}


\begin{abstract} 
In this paper we state some sharp maximum principle, i.e. we characterize  the geometry of the sets of minima for supersolutions of equations involving the {\em $k$-th fractional truncated Laplacian} or the {\em $k$-th fractional eigenvalue}  which are fully nonlinear integral operators whose nonlocality is somehow $k$-dimensional. 
\end{abstract}

\maketitle
\section{Introduction}
Bony's sharp maximum principle \cite{Bo} describes the propagation of minima  for supersolutions of degenerate elliptic operators in all the directions that can be reached through the so called subunit vectors that characterize the operator.  In this paper we show that in the nonlocal context, for some fully nonlinear degenerate operators with nonhomogenous diffusion, this is even more true.

We shall investigate two families of operators  that we now introduce: given a direction $\xi\in\mathbb R^N$ and a function $u:\mathbb R^N\mapsto\mathbb R$, smooth enough, let
\begin{equation}\label{eq1}
\I_\xi u(x)=C_s\int \limits_{0}^{+\infty}\frac{u(x+\tau\xi)+u(x-\tau\xi)-2u(x)}{\tau^{1+2s}}\,d\tau
\end{equation}
where $s\in(0,1)$ and $C_s$ is a normalizing constant such that $\I_\xi u(x)\to\left\langle D^2u(x)\xi,\xi\right\rangle$ as $s\to1^-$.
The so called {\em $k$-th fractional truncated Laplacian} is defined by
\begin{equation}\label{eq2}
\I^-_ku(x)=\inf_{\left\{\xi_i\right\}_{i=1}^k\in{\mathcal V}_k}\sum_{i=1}^k\I_{\xi_i}u(x)
\end{equation}
where ${\mathcal V}_k$ denotes the family of $k$-dimensional orthonormal sets in $\mathbb R^N$, while the {\em $k$-th fractional eigenvalue}, is given by
\begin{equation}\label{eq3}
\I_ku(x)=\inf_{\dim V=k}\sup_{\xi\in V,\;|\xi|=1}\,\I_\xi u(x).
\end{equation}

As it was already pointed out in \cite{BGS,BGT}, these operators do not satisfy the strong maximum principle in the sense that there exists nonnegative supersolutions that reach their 
minimum. More precisely, if $\Omega$ is an open subset of $\mathbb R^N$ and $k<N$, then there exist nonnegative viscosity supersolutions $u\in LSC(\mathbb R^N)$ of 
\begin{equation}\label{eq7}
\I^-_ku=0\quad\text{in $\Omega$},
\end{equation}
or of
$$
\I_k u=0\quad\text{in $\Omega$},
$$
verifying the conditions
$$
\min_\Omega u=0\quad\text{and}\quad U\cap\Omega\neq\emptyset,
$$
where $U=\left\{x\in\R^N\,:\;u(x)>0\right\}$ is the positivity set of $u$.

The main purpose of this paper is  to give a characterization of $U$, for both classes of operators.

Before describing the results we obtain, we wish to mention that we are in the context of nonlocal viscosity solutions, classical references for the definition of viscosity solutions are \cite{BI,CS}; see also \cite{BGT} for a discussion on the notion of viscosity solutions for operators \eqref{eq2} and \eqref{eq3}. It is important to underline the fact that for those operators the nonlocal diffusion is along $k$ one dimensional sets.
Nonlocal operators with nonlocality in sub dimensional sets have been consider in other contests;
in particular Bass and Chen in \cite{BC}, using a probabilistic approach, prove some H\"older regularity results for solutions of equations involving operators that are sums of operators like \eqref{eq1} along $N$ orthogonal directions. Interestingly the regularity of the solution holds even if Harnack's inequality does not. Instead Endal, Ignat and Quiros  in \cite{EIQ} consider the heat equation for a very large class of diffusion terms that include nonlocal operators such as those of Bass and Chen.

We now describe the sharp maximum principle we obtain in the case of the nonlocal truncated Laplacians \eqref{eq2}.  We will prove in Theorem \ref{th1} that if $U$ is the positivity set of a supersolution of \eqref{eq7}, then $U$ satisfies the 
following geometric condition $G_{\Omega,U,k}$.
\begin{defi}
\begin{equation*}
\begin{split}
\text{$U$ satisfies $G_{\Omega,U,k}$} \Leftrightarrow&\;\Omega\backslash U\neq\emptyset\;\;\text{and}\;\;\forall x\in\Omega\backslash U\;\; \text{there exists} \;\left\{\xi_i\right\}_{i=1}^k\in{\mathcal V}_k\;\;\text{s.t}\\&\;x+\tau\xi_i\notin U\;\;\forall\tau\in\mathbb R,\;i=1,\ldots,k.
\end{split}
\end{equation*}
\end{defi}
Furthermore, the \lq\lq viceversa\rq\rq\ is true in the sense that for any $\Omega$ and $U$, open subsets of $\mathbb R^N$, satisfying $G_{\Omega,U,k}$ it is possible to construct a supersolution whose positivity set coincide with $U$, see Theorems \ref{th2} and \ref{th3}.
These two properties somehow reflects the fact that the diffusion is along $k$ orthogonal lines.

In order to give a better understanding of this geometric condition, we show in Theorem \ref{th4} that, if $\partial U$ is smooth, condition 
$G_{\Omega,U,k}$ implies that the sum of the largest $k$ principal curvatures at any point $x\in \Omega\cap\partial U$ is non negative.

In the case of the operator \eqref{eq2}, i.e. $k$-th fractional eigenvalue, the geometric condition of the positivity set is $\mathcal{G}_{\Omega,U,k}$.
\begin{defi}\label{def2}
\begin{equation*}
\begin{split}
\text{$U$ satisfies $\mathcal{G}_{\Omega,U,k}$} \Leftrightarrow&\;\Omega\backslash U\neq\emptyset\;\;\text{and}\;\;\forall x\in\Omega\backslash U\;\; \text{there exists $V$ linear space s.t.} \\
&\; (i)\;\;\dim V=k\\
&\; (ii)\;\; x+v\notin U\;\;\;\forall v\in V\,.
\end{split}
\end{equation*}
\end{defi}
Mutatis mutandis the results are the same as for $ \I^-_k$.
This condition, which is in fact stronger than $G_{\Omega,U,k}$ as it will be seen later, is different in nature because even though the operator is defined only through one dimensional integrals and so with  one dimensional diffusion, the sets involved in the propagation of the minimum are $k$-dimensional affine spaces.

Also, if $\partial U$ is smooth,  $\mathcal{G}_{\Omega,U,k}$ implies that the $(N-k)$-th principal curvature is nonnegative. Furthermore, when $k=N-1$, $\mathcal{G}_{\R^{N-1},U,N-1}$ implies that the connected components of $U$  are convex sets, see Theorems \ref{th8} and \ref{th9}. 

It is inevitable to compare the results that hold in the nonlocal case with those obtained for the differential operators to which these operators tend when $s$ goes to 1.
It is easy to see that, under reasonable conditions on the function $u$, $\I^-_ku$ converges pointwise to ${\mathcal P}^-_k(D^2u):=\lambda_1(D^2u)+\ldots+\lambda_k(D^2u)$, while $\I_ku$ converges to $\lambda_k(D^2u)$, where $\lambda_1(D^2u)\leq\lambda_2(D^2u)\leq\ldots\leq\lambda_N(D^2u)$ denote the eigenvalues of the Hessian $D^2u$ arranged in nondecreasing order. Among other things, for these very degenerate elliptic operators, some characterizations of the positivity set were given  in \cite{BGI}.

Remarkably, the conditions on the principal curvatures of $\partial U$, which are implied by $G_{\Omega,U,k}$ or by  ${\mathcal G}_{\Omega,U,k}$, are the same in the local case and in the nonlocal case. But differently from the local case, see \cite[Theorem 11]{BGI}, the statement of Theorem \ref{th9} cannot be reversed. There exist open sets, with convex connected components, for which ${\mathcal G}_{\R^N,U,N-1}$ is not true. A very simple example  is given by the union of disjoint open balls in $\R^N$.  

 With a different aim, but with a similar point of view, Del Pezzo, Quaas and Rossi in \cite{DQR} have introduced a notion of fractional convexity and they define the $s$-convex envelope through ${\mathcal I}_1u$, in similarity with the result of Oberman and Silvestre \cite{OS}.

\section{The case of $\I^-_k$}
Let $U$ and $\Omega$ be open subsets of $\R^N$. Since $G_{\Omega,U,k}$ is trivial  if $U=\emptyset$, henceforth we shall assume that $U$ is nonempty.

\begin{teo}\label{th1}
Assume that $u\in LSC(\R^N)$ is a bounded and nonnegative (in $\R^N$) viscosity supersolution of 
$$
\I^-_ku=0\quad\text{in $\Omega$}
$$
such that $\displaystyle\min_\Omega u=0$. Then its positivity set $U=\left\{x\in\R^N\,:\;u(x)>0\right\}$ satisfies $G_{\Omega,U,k}$.
\end{teo}
\begin{proof}
Let $x_0\in\Omega\backslash U$. Then $u(x_0)=0$. We test $u$ from below at $x_0$ by
\begin{equation}\label{24mageq2}
\psi_n(x)=\begin{cases}
0 & \text{if $|x-x_0|<\frac1n$}\\
u(x) & \text{otherwise}.\end{cases}
\end{equation}
From the inequality
$$
\I^-_k\psi_n(x_0)\leq0
$$
we infer that, for any $n\in\mathbb N$, there exists an orthonormal frame $\left\{\xi_1(n),\ldots,\xi_k(n)\right\}$ such that 
\begin{equation}\label{eq4}
C_s\sum_{i=1}^k\int\limits_{0}^{+\infty}\frac{\left[u(x_0+\tau\xi_i(n))+u(x_0-\tau\xi_i(n))\right]\chi_{\left(\frac1n,+\infty\right)}(\tau)}{\tau^{1+2s}}\,d\tau\leq\frac1n\,,
\end{equation}
where $\chi_{\left(\frac1n,+\infty\right)}$ denotes the characteristic function of the interval $\left(\frac1n,+\infty\right)$.\\
Up to a subsequence we can further assume that 
$$
\lim_{n\to+\infty}\xi_i(n)=\bar\xi_i\quad\text{for $i=1,\ldots,k$}
$$
and that $\left\{\bar\xi_1,\ldots,\bar\xi_k\right\}\in{\mathcal V}_k$.\\
Using  Fatou's lemma in \eqref{eq4} we obtain
\begin{equation}\label{eq5}
\sum_{i=1}^k\int\limits_{0}^{+\infty}\liminf_{n\to+\infty}\frac{\left[u(x_0+\tau\xi_i(n))+u(x_0-\tau\xi_i(n))\right]\chi_{\left(\frac1n,+\infty\right)}(\tau)}{\tau^{1+2s}}\,d\tau\leq0.
\end{equation}
Since for any $\tau>0$ we have $\chi_{\left(\frac1n,+\infty\right)}(\tau)\to\chi_{\left(0,+\infty\right)}(\tau)$ as $n\to+\infty$   and 
$$
u(x_0\pm\tau\bar\xi_i)\leq\liminf_{n\to+\infty}u(x_0\pm\tau\xi_i(n)),
$$
by lower semincontinuity, then we infer from \eqref{eq5} that
\begin{equation}\label{eq6}
\sum_{i=1}^k\int\limits_{0}^{+\infty}\frac{u(x_0+\tau\bar\xi_i)+u(x_0-\tau\bar\xi_i)}{\tau^{1+2s}}\,d\tau\leq0.
\end{equation}
Moreover $u\geq0$ in $\mathbb R^N$ by assumption, hence by \eqref{eq6} we conclude that
  $$u(x_0+\tau\bar\xi_i)=0\quad\forall \tau\in\R,\;i=1,\ldots,k$$ 
that is $x_0+\tau\bar\xi_i\notin U$ for any $\tau\in\R$ and any $i=1,\ldots,k$.
\end{proof}
\begin{remark}{\rm  Theorem \ref{th1} excludes the existence of bounded supersolutions $u\in LSC(\R^N)$ of $$\I^-_ku=0 \quad \text{in $\Omega$}$$ such that $$\min_\Omega u=0\quad\;\text{and}\quad\;u>0\;\;\text{in $\R^N\backslash\Omega$}.$$}

\end{remark}

Conversely, any open subset $U$ satisfying $G_{\Omega,U,k}$ coincide with the positivity set of a supersolution $u$ of the equation \eqref{eq7}.  

\begin{teo}\label{th2}
Assume that $U$ is an open set satisfying  $G_{\Omega,U,k}$. Then there exists a nonnegative and bounded (in $\R^N$) supersolution $u\in LSC(\R^N)$ of 
$$
\I^-_ku=0\quad\text{in $\Omega$}
$$
such that 
$$
\min_\Omega u=0\qquad\text{and}\qquad U=\left\{x\in\R^N\,:\;u(x)>0\right\}.
$$
\end{teo}
\begin{proof}
Let us define $u\in LSC(\R^N)$ by
$$
u(x)=\chi_U(x).
$$
It is clear that $U=\left\{x\in\R^N\,:\;u(x)>0\right\}$. Moreover, since $\Omega\backslash U\neq\emptyset$ by definition of $G_{\Omega,U,k}$, it holds that $\displaystyle\min_\Omega u=0$. \\
To prove that $u$ is a viscosity supersolution, let $x_0\in\Omega$. If $x_0\in \Omega\cap U$, then $u(x_0)=1$.  Since $U$ is open, the function $u$ is in fact constant in a neighborhood of $x_0$. Hence, for any $\xi \in\R^N$, with $|\xi|=1$,  $\I_\xi u(x_0)$ is well defined and 
$$
\I_\xi u(x_0)=C_s\int \limits_{0}^{+\infty}\frac{u(x+\tau\xi)+u(x-\tau\xi)-2}{\tau^{1+2s}}\,d\tau\leq0
$$ 
considering that  $0\leq u(x)\leq1$ for any $x\in\R^N$. Then we infer that the inequality
$$
\I^-_ku(x_0)\leq0
$$
holds in classical, and so in the viscosity, sense at $x_0\in \Omega\cap U$. \\
Assume now that $x_0\in\Omega\backslash U$, so that $u(x_0)=0$. In this case $u$ can be discontinuous at $x_0$ depending on whether $x_0\in\Omega\cap\partial U$ or not.\\
 To check that the inequality $\I^-_ku(x_0)\leq0$ holds in the viscosity sense, let $\varphi\in C^2\left(\overline{B_\delta(x_0)}\right)$, $\delta>0$, be such that 
\begin{equation}\label{eq8}
\varphi(x_0)=0\qquad\text{and}\qquad\varphi (x)\leq u(x)\;\;\forall x\in B_\delta(x_0).
\end{equation}
Consider the function
$$
\psi(x)=\begin{cases}
\varphi(x) & \text{if $x\in B_\delta(x_0)$}\\
u(x) & \text{otherwise}.
\end{cases}
$$
Using the hypothesis $G_{\Omega,U,k}$, there exists an orthonormal frame $\left\{\bar\xi_1,\ldots,\bar\xi_k\right\}$, depending on $x_0$, such that 
\begin{equation}\label{eq9}
u(x_0+\tau\bar\xi_i)=0\quad\forall \tau\in\R,\;i=1,\ldots,k.
\end{equation}
Thus, by \eqref{eq8}-\eqref{eq9}, we conclude 
\begin{equation*}
\I^-_k\psi(x_0)\leq\sum_{i=1}^k\I_{\bar\xi_i}\psi(x_0)
=C_s \sum_{i=1}^k\int \limits_{0}^{\delta}\frac{\varphi(x+\tau\bar\xi_i)+\varphi(x-\tau\bar\xi_i)}{\tau^{1+2s}}\,d\tau\leq0\,.
\end{equation*}
\end{proof}

\begin{remark}{\rm
As a consequence of Theorem \ref{th1} and of the proof of Theorem \ref{th2},  if $u\in LSC(\R^N)$ is a nonnegative and bounded (in $\R^N$) viscosity supersolution of $$\I^-_ku\leq0\quad\text{in $\Omega$,}$$ then the characteristic function $\chi_U$ of its positivity set $U=\left\{x\in\R^N\,:\;u(x)>0\right\}$ is in turn viscosity supersolution of the same equation, that is 
$$
\I^-_k\chi_U\leq0\quad\text{in $\Omega$.}
$$ }
\end{remark}

Next Theorem  provides the existence of a Lipschitz continuous  entire supersolution with a prescribed positivity set $U$ satisfying the property $G_{\R^N,U,k}$.

\begin{teo}\label{th3}
Assume that $U$ is an open and bounded subset of $\R^N$ that satisfies $G_{\R^N,U,k}$. Then there exists a nonnegative and bounded viscosity supersolution $u\in\mathrm{Lip}(\R^N)$ of 
\begin{equation}\label{20mageq1}
\I^-_ku=0\quad\text{in $\R^N$}
\end{equation}
such that $U=\left\{x\in\R^N\,:\;u(x)>0\right\}$.
\end{teo}
\begin{proof}
We define
$$
u(x)=\text{dist}(x,\R^N\backslash U).
$$
Clearly $u$ is bounded and Lipschitz in $\R^N$. Moreover, since $U$ is open, $u(x)>0$ if, and only if, $x\in U$.\\
To check that $u$ is a viscosity supersolution of \eqref{20mageq1}, let $x_0\in\R^N$ and let $\varphi\in C^2\left(\overline{B_\delta(x_0)}\right)$, $\delta>0$, be such that 
\begin{equation}\label{20mageq2}
u(x_0)-\varphi(x_0)=0\leq u(x)-\varphi(x)\qquad\forall x\in B_\delta(x_0).
\end{equation}
Setting
$$
\psi(x)=\begin{cases}
\varphi(x) & \text{if $x\in B_\delta(x_0)$}\\
u(x) & \text{otherwise},
\end{cases}
$$
we have to prove that $\I^-_k\psi(x_0)\leq0$.\\
We choose $y_0\in\R^N\backslash U$, depending on $x_0$, such that
\begin{equation}\label{20mageq3}
\varphi(x_0)=u(x_0)=|x_0-y_0|\,.
\end{equation}
Moreover, by the definition of $u$, we have
\begin{equation}\label{20mageq4}
\varphi(x)\leq u(x)\leq |x-y|\quad\forall x\in B_\delta(x_0),\;y\in\R^N\backslash U. 
\end{equation}
Using \eqref{20mageq3}-\eqref{20mageq4} with $x=y+x_0-y_0$ and setting $\phi(y)=\varphi(y+x_0-y_0)$, we then obtain 
\begin{equation}\label{20mageq5}
\phi(y)\leq \phi(y_0)\quad\forall y\in B_\delta(y_0)\cap(\R^N\backslash U). 
\end{equation}
Since $y_0\in\R^N\backslash U$, by the assumption $G_{\R^N,U,k}$ there exists an orthonormal frame $\left\{\bar\xi_1,\ldots,\bar\xi_k\right\}$ such that 
\begin{equation}\label{20mageq6}
y_0+\tau\bar\xi_i\notin U\;\;\;\forall\tau\in\R,\;i=1,\ldots,k.
\end{equation}
Then, from \eqref{20mageq5}-\eqref{20mageq6}, we have 
$$
\phi(y_0+\tau\bar\xi_i)+\phi(y_0-\tau\bar\xi_i)-2\phi(y_0)\leq0\quad\forall\tau\in[0,\delta),\;i=1,\ldots,k.
$$
The above inequality implies that 
$$
\sum_{i=1}^k\int \limits_{0}^{\delta}\frac{\phi(y_0+\tau\bar\xi_i)+\phi(y_0-\tau\bar\xi_i)-2\phi(y_0)}{\tau^{1+2s}}\,d\tau\leq0\,.
$$
Moreover, since $\phi(y_0\pm \tau\bar\xi_i)=\varphi(x_0\pm\tau\bar\xi_i)$ for $\tau\in[0,\delta)$ and $i=1,\ldots,k$, we obtain
\begin{equation}\label{20mageq7}
\sum_{i=1}^k\int \limits_{0}^{\delta}\frac{\varphi(x_0+\tau\bar\xi_i)+\varphi(x_0-\tau\bar\xi_i)-2\varphi(x_0)}{\tau^{1+2s}}\,d\tau\leq0\,.
\end{equation}
Now we use the inequality
$$
u(x)\leq |x-y|\quad\forall x\in\R^N,\;y\in\R^N\backslash U
$$
with the particular choice $x=x_0\pm\tau\bar\xi_i$ and $y=y_0\pm\tau\bar\xi_i$ to infer that 
$$
u(x_0\pm\tau\bar\xi_i)\leq|x_0-y_0|=u(x_0) \quad\forall \tau\in[0,+\infty),\;i=1,\ldots,k\,.
$$
Thus 
\begin{equation}\label{20mageq8}
\sum_{i=1}^k\int \limits_{\delta}^{+\infty}\frac{u(x_0+\tau\bar\xi_i)+u(x_0-\tau\bar\xi_i)-2u(x_0)}{\tau^{1+2s}}\,d\tau\leq0\,.
\end{equation}
The conclusion $$\I^-_k\psi(x_0)\leq0$$ easily follows from \eqref{20mageq7}-\eqref{20mageq8}.
\end{proof}

\begin{remark}\label{rmk2}
{\rm
The assumption $U$ bounded in the statement of Theorem \ref{th3} has been used only to guarantee that $u(x)=\text{dist}(x,\R^N\backslash U)$ was bounded. In this way the maps
\begin{equation}\label{26mageq1}
\tau\mapsto\frac{u(x+\tau\xi)+u(x-\tau\xi)-2u(x)}{\tau^{1+2s}}
\end{equation}
are integrable outside the origin, for any direction $\xi$ and any $x\in\R^N$.\\
In the case $U$  unbounded, using the Lipschitz continuity of $u$, the integrability of \eqref{26mageq1} far away from the origin, is still true, independently of $\xi$, provided $s>\frac12$.\\
Another possibility to deal with general $U$ and without the restriction $s\in(\frac12,1)$, is to replace $u(x)=\text{dist}(x,\R^N\backslash U)$ by
$$
u(x)=\min\left\{\text{dist}(x,\R^N\backslash U),1\right\}.
$$
The details are left to the reader.
}
\end{remark}

Let $u$ be a viscosity supersolution of 
\begin{equation}\label{23mageq1}
\I^-_ku=0\quad\text{in $\Omega$}
\end{equation}
and suppose that the relative boundary $\Omega\cap\partial U$ of its positivity set $U$ is smooth. Then, denoting with
$$
\kappa_1(x)\leq\ldots\leq\kappa_{N-1}(x)
$$
the principal curvatures of $\Omega\cap\partial U$ at $x$, one has 
\begin{equation*}
\kappa_{N-k}(x)+\ldots+\kappa_{N-1}(x)\geq0 \quad\forall x\in\Omega\cap\partial U.
\end{equation*}
This geometric property is a consequence of Theorem \ref{th1} and the following

\begin{teo}\label{th4}
Let $U$ be an open set verifying $G_{\Omega,U,k}$ and assume that $\Omega\cap\partial U$ 
is a $C^2$-hypersurface. 
Then
\begin{equation}\label{23mageq2}
\sum_{i=1}^k\kappa_{N-i}(x)\geq0 \quad\forall x\in\Omega\cap\partial U.
\end{equation}
\end{teo}
\begin{proof}
Let $x_0\in\Omega\cap\partial U$. By assumption, and by an orthogonal transformation, we may assume that $x_0=0$ and that  for some $r>0$
\begin{equation}\label{23mageq4}
U\cap B_r=\left\{x\in B_r\,:\;x_N>f(x')\right\}, 
\end{equation}
where $x=(x',x_N)$ and 
\begin{equation}\label{24mageq1}
f\in C^2(B_r(x')),\quad f(0')=0, \quad Df(0')=0'.
\end{equation}
 Moreover the principal curvatures of $\Omega\cap\partial U$ at $x_0=0$ are the eigenvalues of $D^2f(0')$.\\
Since $G_{\Omega,U,k}$ holds and $x_0=0\in\Omega\backslash U$, then there exists $\left\{\bar\xi_i\right\}_{i=1}^k\in{\mathcal V}_k$ such that 
\begin{equation}\label{23mageq3}
\tau\bar\xi_i\notin U\;\;\quad\forall\tau\in\mathbb R,\;i=1,\ldots,k.
\end{equation}
We claim that $\left\langle \bar\xi_i,e_N\right\rangle=0$ for any $i=1,\ldots,k$. \\ If not, then $\left\langle \bar\xi_i,e_N\right\rangle\neq0$ for some $i\in\left\{1,\ldots,k\right\}$. Replacing $\bar\xi_i$ with $-\bar\xi_i$ if necessary, we can further suppose that $\left\langle \bar\xi_i,e_N\right\rangle>0$. Since
$$
f\left(\tau\bar\xi'_i\right)=o(\tau)\quad\text{as $\tau\to0$, }
$$
we infer that for any $\tau$ positive and small enough
$$
\tau\left\langle \bar\xi_i,e_N\right\rangle> f\left(\tau\,\bar\xi'_i\right)\,.
$$
Thus, using \eqref{23mageq4}, we have that  for any $\tau$ positive and small enough
$$\tau\bar\xi_i\in U$$
which contradicts  \eqref{23mageq3}.\\
Since $\left\langle \bar\xi_i,e_N\right\rangle=0$, we can write 
\begin{equation}\label{23mageq5}
\bar\xi_i=(\bar\xi'_i,0).
\end{equation} 
Moreover $\left\{\bar\xi'_1,\ldots,\bar\xi'_k\right\}$ is an orthonormal frame in $\R^{N-1}$.\\
Consider now, for  $i=1,\ldots,k$, the functions
$$g_i(\tau)=f(\tau\bar\xi'_i)\qquad\tau\in(-r, r).$$ 
  Using \eqref{23mageq4}-\eqref{24mageq1}-\eqref{23mageq3}-\eqref{23mageq5} we obtain that 
$$
g_i(\tau)\geq0=g(0)\qquad\forall \tau\in(-r,r).
$$
Hence for any $i=1,\ldots,k$ 
\begin{equation*}
g_i''(0)=\left\langle D^2f(0')\bar\xi'_i,\bar\xi'_i \right\rangle\geq0,
\end{equation*}
from which we conclude that 
\begin{equation*}
\begin{split}
\sum_{i=1}^k\kappa_{N-i}(x_0)&=\sup\left\{\sum_{i=1}^k\left\langle D^2f(0')\xi_i,\xi_i\right\rangle\;:\;\;\left\{\xi_1,\ldots,\xi_k\right\} \text{orthonormal set in $\R^{N-1}$}\right\}\\
&\geq\sum_{i=1}^k\left\langle D^2f(0')\bar\xi'_i,\bar\xi'_i\right\rangle\geq0.
\end{split}
\end{equation*}

\end{proof}

\section{The case of $\I_k$}
As it was mentioned in the introduction, for the operator $\I_k$ the right condition on the positivity set is $\mathcal{G}_{\Omega,U,k}$, see Definition \ref{def2}. We start the section giving an example
of a bounded open set $U\subset\mathbb R^3$ which satisfies the condition $G_{\mathbb R^3,U,2}$, but not $\mathcal{G}_{\mathbb R^3,U,2}$. Such example can be easily generalized to $\mathbb R^N$ and $k<N$. 
\begin{example}{\rm
Let
$$
U=U_1\cup U_2,
$$
where $U_1$ and $U_2$ are the two unit and open balls defined by 
\begin{equation*}
\begin{split}
U_1&=\left\{(x,y,z)\in\mathbb R^3\,:\;x^2+y^2+z^2<1\right\}\\
U_2&=\left\{(x,y,z)\in\mathbb R^3\,:\;x^2+(y-4)^2+z^2<1\right\}.
\end{split}
\end{equation*}
Let us first show that condition $\mathcal{G}_{\mathbb R^3,U,2}$ is not satisfied. For this we consider, for $\varepsilon$ positive and small enough,  the point $P_\varepsilon=(0,\varepsilon,\sqrt{1-\varepsilon^2})\notin U$. We claim that every two dimensional plane $\pi$ passing through $P_\varepsilon$ has nonempty intersection with $U$. This is obvious if $\pi$ is not the tangent plane to the unit sphere $\partial U_1=\left\{x^2+y^2+z^2=1\right\}$, since $P_\varepsilon\in\partial U_1$. On the other hand if  $\pi$ is the tangent plane to $\partial U_1$ at the point $P_\varepsilon$, it is not difficult to see that 
$$
\left(0,4,\sqrt{1-\varepsilon^2}-\frac{\varepsilon}{\sqrt{1-\varepsilon^2}}(4-\varepsilon)\right)\in\pi\cap U_2
$$
for any $\varepsilon$ positive and sufficiently small. Hence $\mathcal{G}_{\mathbb R^3,U,2}$ is not fulfilled at $P_\varepsilon$.

Now we prove that $U$ has the property ${G}_{\mathbb R^3,U,2}$. Fix $(x_0,y_0,z_0)\notin U$. 
If there exists a two-dimensional linear space $V\subset \R^3$ such that $\left((x_0,y_0,z_0)+V\right)\cap U=\emptyset,$ then we are done. 
Suppose now that for any linear space $V$, with $\dim V=2$, $$\left((x_0,y_0,z_0)+V\right)\cap U\neq\emptyset.$$ 
In particular, this implies that 
\begin{equation}\label{2806eq1}
x_0\in(-1,1)\;,\quad y_0\in(0,1)\cup(3,4)\;,\quad z_0\in(-1,1)\,.
\end{equation}
Otherwise, it is obvious that there exists a two-dimensional linear space $W$ such that $$\left((x_0,y_0,z_0)+W\right)\cap U=\emptyset\,.$$
By symmetry, we suppose in the following that $y_0\in(0,1)$. A similar argument, with obvious changes, holds in the case $y_0\in(3,4)$
We can pick $V$ in such a way $$\left((x_0,y_0,z_0)+V\right)\cap U_1=\emptyset\,. $$ 
We then have that 
$\left((x_0,y_0,z_0)+V\right)\cap U_2\neq\emptyset\,.$ Note that $((x_0,y_0,z_0)+V)\cap U_2$ is a 
disk and that 
$$\text{dist}\left((x_0,y_0,z_0),((x_0,y_0,z_0)+V)\cap U_2\right)\geq |(x_0,y_0,z_0)-(0,4,0)|-1>2\,.$$ 
After a rigid transformation we can also assume that $(x_0,y_0,z_0)=(0,0,0)$, $V$ is spanned by $e_1=(1,0,0)$ and $e_2=(0,1,0)$ and that $$\left((x_0,y_0,z_0)+V\right)\cap U_2=\left\{(x,y,0)\in\mathbb R^3\,:\;(x-\alpha)^2+y^2<\rho^2\right\}$$
for some $\alpha>2$ and $\rho\in(0,1]$.  
Now, it is immediate to see that 
$$
r_1:\begin{cases}
x=y \\
z=0
\end{cases}\;,\qquad r_2:\begin{cases}
x=-y \\
z=0
\end{cases}
$$
are two orthogonal lines passing through $(x_0,y_0,z_0)$ and such that $$
r_1\cap U=r_2\cap U=\emptyset.
$$
Since $(x_0,y_0,z_0)\notin U$ is arbitrary, we conclude that $U$ has the property ${G}_{\mathbb R^3,U,2}$.
}
\end{example}

\begin{teo}\label{th5}
Assume that $u\in LSC(\R^N)$ is a bounded and nonnegative (in $\R^N$) viscosity supersolution of 
$$
\I_ku=0\quad\text{in $\Omega$}
$$
such that $\displaystyle\min_\Omega u=0$. Then its positivity set $U=\left\{x\in\R^N\,:\;u(x)>0\right\}$ satisfies $\mathcal{G}_{\Omega,U,k}$.
\end{teo}
\begin{proof}
Let $x_0\in\Omega\backslash U$. We proceed as in the proof of Theorem \ref{th1} testing $u$ from below at $x_0$ by the sequence of function $\psi_n$ defined in \eqref{24mageq2}. Hence for any $n\in\mathbb N$, there exists a $k$-dimensional linear space $V=V(n)$ such that 
\begin{equation}\label{24mageq3}
C_s\sup_{\xi\in V,\;|\xi|=1}\int\limits_{0}^{+\infty}\frac{\left[u(x_0+\tau\xi)+u(x_0-\tau\xi)\right]\chi_{\left(\frac1n,+\infty\right)}(\tau)}{\tau^{1+2s}}\,d\tau\leq\frac1n\,.
\end{equation}
Choose an orthonormal basis $\left\{\xi_1(n),\ldots,\xi_k(n)\right\}$ of $V$. Passing to a subsequence, we can also assume that 
$$
\lim_{n\to+\infty}\xi_i(n)=\bar\xi_i\quad\;i=1,\ldots,k.
$$
Let $\bar V=\text{span}\left\{\bar\xi_1,\ldots,\bar\xi_k\right\}$. We have $\dim \bar V=k$. We claim that $\bar V$ satisfies also condition $(ii)$ in Definition \ref{def2}. For this we shall prove that for any fixed unit vector $\xi\in\bar V$, then
\begin{equation}\label{24mageq4}
x_0+\tau\xi\notin U\quad\;\forall \tau\in\R.
\end{equation}
Let $\xi\in\bar V$ be such that $|\xi|=1$ and let 
$$
v(n):=\sum_{i=1}^k\left\langle \xi_i(n),\xi\right\rangle\xi_i(n).
$$
Note that $v(n)\in V=V(n)$ and that 
\begin{equation}\label{24mageq5}
\lim_{n\to+\infty}v(n)=\xi.
\end{equation}
Hence, for any $n$ sufficiently large, $|v(n)|>0$.
Using \eqref{24mageq3} with $\xi(n)=\frac{v(n)}{|v(n)|}\in V$ we have 
$$
C_s\int\limits_{0}^{+\infty}\frac{\left[u(x_0+\tau\xi(n))+u(x_0-\tau\xi(n))\right]\chi_{\left(\frac1n,+\infty\right)}(\tau)}{\tau^{1+2s}}\,d\tau\leq\frac1n\,.
$$
Passing to the limit  in the above inequality as $n\to+\infty$ and using Fatou's Lemma, $u\in LSC(\R^N)$ and \eqref{24mageq5} we obtain
$$
\int\limits_{0}^{+\infty}\frac{u(x_0+\tau\xi)+u(x_0-\tau\xi)}{\tau^{1+2s}}\,d\tau\leq0\,.
$$
The assumption $u\geq0$ in $\R^N$ yields $u(x_0+\tau\xi)=0$ for any $\tau\in\R$. Thus \eqref{24mageq4} holds.
\end{proof}

The proofs of the following three theorems can be obtained, with minor changes, as those of Theorems \ref{th2}-\ref{th3}-\ref{th4}, see also Remark \ref{rmk2}. 

\begin{teo}\label{th6}
Assume that $U$ is an open set satisfying  ${\mathcal G}_{\Omega,U,k}$. Then there exists a nonnegative and bounded (in $\R^N$) supersolution $u\in LSC(\R^N)$ of 
$$
\I_ku=0\quad\text{in $\Omega$}
$$
such that 
$$
\min_\Omega u=0\qquad\text{and}\qquad U=\left\{x\in\R^N\,:\;u(x)>0\right\}.
$$
\end{teo}

\begin{teo}\label{th7}
Assume that $U$ is an open subset of $\R^N$ that satisfies ${\mathcal G}_{\R^N,U,k}$. Then there exists a nonnegative and bounded viscosity supersolution $u\in\mathrm{Lip}(\R^N)$ of 
\begin{equation*}
\I_ku=0\quad\text{in $\R^N$}
\end{equation*}
such that $U=\left\{x\in\R^N\,:\;u(x)>0\right\}$.
\end{teo}

\begin{teo}\label{th8}
Let $U$ be an open set verifying ${\mathcal G}_{\Omega,U,k}$ and assume that $\Omega\cap \partial U$ is a smooth hypersurface. 
Then
\begin{equation*}
\kappa_{N-k}(x)\geq0 \quad\forall x\in\Omega\cap\partial U.
\end{equation*}
\end{teo}

The geometric condition ${\mathcal G}_{\R^N,U,N-1}$ implies the convexity of any connected component of $U$, as shown in the next

\begin{teo}\label{th9}
Let $U\subset\R^N$ be open set satisfying condition ${\mathcal G}_{\R^N,U,N-1}$. Then any connected component of $U$ is a convex set.
\end{teo}
\begin{proof} Let $U_0$ be a connected component of $U$. Note that $U_0$ is an open subset 
of $U$ since $U$ is locally connected. 

To the contrary that 
$U_0$ is a convex set, we suppose that $U_0$ is not convex, which implies that 
there are two distinct points $x,y\in U_0$ such that  
the line segment $[x,y]:=\{tx+(1-t)y : t\in[0,1]\}$ is  not contained in $U_0$. 
This implies that $[x,y]$ is not contained in $U$. Indeed, if we suppose that 
$[x,y]\subset U$, then $[x,y]$ is a connected subset of $U$ and intersects $U_0$, 
and therefore, $U_0\cup [x,y]$ is a connected subset of $U$, which yields a contradiction  that 
$[x,y]\subset U_0$. Thus, we find that $[x,y]$ is not contained in $U$. We can choose 
$\lambda\in(0,1)$ so that $z:=\lambda x+(1-\lambda)y\not\in U$. By the property ${\mathcal G}_{\R^N,U,N-1}$,
we can choose $(N-1)$-dimensional linear subspace $V$ of $\R^N$ such that 
$(z+V)\cap U=\emptyset$. Select $\nu\in\R^N\setminus\{0\}$ so that 
$z+V=\{p\in\R^N : \langle\nu, p-z\rangle=0\}$ and set 
\[
H^+=\{p\in\R^N : \langle \nu, p-z\rangle>0\} \ \ \text{ and } \ \ H^-=\{p\in\R^N : \langle 
\nu, p-z\rangle<0\}.  
\]
Observe that 
\begin{equation} \label{disjoint-union}
U_0=(H^+\cap U_0)\ \cup\ (H^-\cap U_0),
\end{equation}
that the right hand side of \eqref{disjoint-union} is a disjoint union of two open subsets of $U$,
and that either $x\in H^+$ and $y\in H^-$, or $x\in H^-$ and $y\in H^+$, which 
assures that both $H^+\cap U_0$ and $H^-\cap U_0$ are nonempty. 
Hence, \eqref{disjoint-union} contradicts the connectedness of $U_0$, which
completes the proof. 
\end{proof}
As a consequence of Theorems \ref{th5} and \ref{th9}, we have the following

\begin{teo}\label{th10}
Let $u\in LSC(\R^N)$ be a bounded and nonnegative supersolution of \begin{equation}\label{24mageq8}
\I_{N-1}u\leq0\quad\text{in $\R^N$}
\end{equation} and satisfy $\displaystyle\min_{\Omega}u=0$. Then any connected components of its positivity set is a convex set.
\end{teo}

\begin{remark}{\rm
Theorem \ref{th10} is not true for supersolutions of 
\begin{equation}\label{27giueq1}
\I_ku=0\quad\quad\text{in $\R^N$}
\end{equation}
when $k<N-1$ and $N\geq3$. As example consider 
$$
U=\left\{x=(x_1,\ldots,x_N)\in\R^N\,:\;1< x_1^2+x_2^2<2\right\}.
$$
Then $U$ has the property ${\mathcal G}_{\R^N,U,k}$ for any $k\leq N-2$. For this it is sufficient to consider any $k$-dimensional linear space $V\subseteq\left\{(0,0,x_3,\ldots,x_N)\right\}$ in order to fulfill condition $(ii) $ of Definition \ref{def2}. \\
By Theorem \ref{th6}, or Theorem \ref{th7}, there exists a bounded and nonnegative viscosity supersolution of \eqref{27giueq1} whose positivity set coincides with  $U$. On the other hand it is clear that $U$ is not convex.
}
\end{remark}


\begin{thebibliography}{00}
\bibitem{BI}
Barles, G. and Imbert, C. {\em Second-order Elliptic Integro-Differential Equations: Viscosity Solutions' Theory Revisited.} Ann. Inst. H. Poincar\'e Anal. Non Lin\'eaire, Vol. 25 (2008) no. 3, 567-585.
\bibitem{BC} R. F. Bass, Zhen-Qing Chen,\emph{ Regularity of Harmonic functions for a class of singular stable-like processes}, Math. Z. 266 (2010) 489-503.
\bibitem{BGI} I. Birindelli, G. Galise, H. Ishii, \emph{ Positivity sets of supersolutions of degenerate elliptic equations and the strong maximum principle},   Trans. Amer. Math. Soc.  374 (2021), no. 1, 539-564.
\bibitem{BGS} I. Birindelli, G. Galise, D. Schiera, \emph{Maximum principles and related problems for a class of nonlocal extremal operators} Ann. Mat. Pura Appl. (to appear) 
\bibitem{BGT} I. Birindelli, G. Galise, E. Topp, \emph{Fractional truncated Laplacians: representation formula, fundamental solutions and applications}, Nonlinear Differ. Equ. Appl. 29, 26 (2022). 
\bibitem{Bo}  J.-M. Bony,\emph{ Principe du maximum, in\'egalit\'e de Harnack et unicit\'e du probl\`eme de Cauchy pour les op\'erateurs elliptiques d\'eg\'en\'er\'es}, Ann. Inst. Fourier (Grenoble) 19 (1969) 277-304.
\bibitem{CS}
Caffarelli, L. and Silvestre, L. {\em Regularity Theory For Nonlocal Integro-Differential Equations.} Comm. Pure Appl. Math, Vol. 62 
(2009), no. 5, 597-638.
\bibitem{DQR} L. Del Pezzo, A. Quaas, J. Rossi, \emph{Fractional convexity},  Math. Ann. (to appear). 
\bibitem{EIQ} J. Endal, L. I. Ignat, F. Quiros, \emph{Large-time behaviour for anisotropic stable nonlocal diffusion problems with convection} arXiv 2207.01874v1
\bibitem{OS} A. M. Oberman and L. Silvestre, \emph{The Dirichlet problem for the convex envelope}, Trans. Amer. Math. Soc. 363 (2011), no. 11, 5871-5886.
\end{thebibliography}
\end{document}